\RequirePackage{fix-cm}
\documentclass[smallextended]{svjour3}       
\smartqed  
\usepackage{epsfig,amsfonts,amssymb,latexsym,amsmath}
\usepackage{graphics}
\usepackage{graphicx}
\usepackage{subfigure}
\usepackage{epsfig}
\usepackage{algorithm}
\usepackage{algorithmic}
\usepackage{url}
\usepackage{rotating}

\begin{document}

\title{Spectral Residual Method for Nonlinear Equations on Riemannian Manifolds}


\author{Harry Oviedo         \and
        Hugo Lara 
}


\institute{H. Oviedo \at
              Mathematics Research Center, CIMAT A.C. Guanajuato, Mexico.\\
              \email{harry.oviedo@cimat.mx}           
           \and
           H. Lara \at
              Universidade Federal de Santa Catarina, Campus Blumenau, Blumenau, Brazil.\\
              \email{hugo.lara.urdaneta@ufsc.br}
}

\date{Received: date / Accepted: date}

\maketitle

\begin{abstract}
In this paper, the spectral algorithm for nonlinear equations (SANE) is adapted to
the problem of finding a zero of a given tangent vector field
on a Riemannian manifold. The generalized version of SANE uses,
in a systematic way, the tangent vector field as a search direction and
a continuous real--valued function that adapts this direction and ensures
that it verifies a descent condition for an associated merit function.
In order to speed up the convergence of the proposed method, we incorporate a
Riemannian adaptive spectral parameter in combination with a non--monotone
globalization technique. The global convergence of the proposed procedure
is established under some standard assumptions. Numerical results indicate that our algorithm is very effective and efficient
solving tangent vector field on different Riemannian manifolds and competes
favorably with a Polak--Ribi\'ere--Polyak Method recently published and other methods
existing in the literature.
\end{abstract}

\section{Introduction}
\label{sec:1}
In this work, we consider the problem of finding a zero of a tangent vector field $F(\cdot)$ over a Riemannian manifold $\mathcal{M}$, with an associated Riemannian metric $\langle \cdot,\cdot \rangle$ (a local inner product which induces a corresponding local metric). The problem can be mathematically formulated as the solution of the following nonlinear equation
\begin{equation}
  F(X) = 0_{X}, \label{problem}
\end{equation}
where $F:\mathcal{M}\to \mathcal{TM}$ is a continuously differentiable tangent vector field, and $\mathcal{TM} := \cup_{X\in\mathcal{M}}T_X\mathcal{M}$ denotes the tangent bundle of $\mathcal{M}$, i.e., $\mathcal{TM}$ is the union of all tangent spaces at points in the manifold. Here, $0_{X}$ denotes the zero vector of the tangent space $T_X{\mathcal{M}}$. This kind of problem appears frequently in several applications, for example: statistical principal component analysis \cite{oja1989neural}, where the Oja's flow induces the associated vector field, total energy minimization in electronic structure calculations \cite{martin2020electronic,saad2010numerical,zhang2014gradient}, linear eigenvalue problems \cite{manton2002optimization,wen2016trace}, dimension reduction techniques in pattern recognition \cite{kokiopoulou2011trace,turaga2008statistical,zhang2018robust}, Riemannian optimization problems, where the Riemannian gradient flow leads to the associated tangent vector field \cite{absil2009optimization,edelman1998geometry,ring2012optimization}, among others.\\

Problem \eqref{problem} is closely related to the problem of minimizing a differentiable function over the manifold $\mathcal{M}$,
\begin{equation}
  \min \mathcal{F}(X) \quad \textrm{s.t.} \quad X\in\mathcal{M}, \label{problem1}
\end{equation}
where $\mathcal{F}:\mathcal{M} \to \mathbb{R}$ is a smooth function. Different iterative methods have been developed for solving \eqref{problem1}. Some popular schemes are based on gradient method \cite{cedeno2018projected,iannazzo2018riemannian,manton2002optimization,oviedoimplicit,oviedo2019scaled,oviedo2020two,oviedo2019non}, conjugate gradient methods \cite{absil2009optimization,edelman1998geometry,zhu2017riemannian}, Newton’s method \cite{absil2009optimization,sato2014riemannian}, or quasi--Newton methods \cite{absil2009optimization}. All these numerical methods can be used to find a zero of the following tangent vector field equation,
\begin{equation}
  \nabla_{\mathcal{M}}\mathcal{F}(X) = 0_{X}, \label{grad_field}
\end{equation}
where $\nabla_{\mathcal{M}}\mathcal{F}(\cdot)$ denotes the Riemannian gradient of $\mathcal{F}(\cdot)$, which is a particular case of problem \eqref{problem}. The Riemannian line--search methods, designed to solve the optimization problem \eqref{problem1} construct a sequence of points using the following recursive formula
\begin{equation}
  X_{k+1} = R_{X_k}[\xi_{X_k}], \label{line_search}
\end{equation}
where $R_X:T_X\mathcal{M}\to \mathcal{M}$ is a retraction (see Definition \ref{def1}), and $\xi_{X_k}\in T_{X_k}\mathcal{M}$ is a descent direction, i.e, $\xi_{X_k}$ verifies the inequality $\langle\nabla_\mathcal{M}\mathcal{F}(X_k),\xi_{X_k}\rangle < 0$ for all $k\geq 0$. Among the Riemannian line--search methods, the Riemannian gradient approach exhibits the lowest cost per iteration. This method uses the gradient vector field $\nabla_{\mathcal{M}}\mathcal{F}(\cdot)$ to define the search direction by $\xi_{X_k} = -\nabla_{\mathcal{M}}\mathcal{F}(X_k)$, at each iteration.\\

In the literature, there are some iterative algorithms addressing the problem \eqref{problem}. In \cite{adler2002newton,dedieu2003newton,li2005convergence} were developed several Riemannian Newton methods for the solution of tangent vector fields on general Riemannian manifolds. Among the features of Newton's method, the requirement of using second--order information and geodesics (that involves the computation of exponential mapping) to ensure keeping into the corresponding manifold, leads to the growth of computational cost. In addition, the authors in \cite{breiding2018convergence} proposed a Riemannian Gauss--Newton method to address the solution of \eqref{problem} through the optimization problem \eqref{problem1}. Recently, in \cite{yao2020riemannian} was introduced a Riemannian conjugate gradient to deal with the numerical solution of \eqref{problem}, that does not need derivative computation (it does not use the Jacobian of $F(\cdot)$), and that incorporates the use of retractions (see Definition \ref{def1}), which is a mapping that generalizes the definition of geodesics, and that was introduced by Absil in \cite{absil2009optimization} to deal with optimization problems on matrix manifolds.\\

In the Euclidean case $\mathcal{M} = \mathbb{R}^{n}$ in \eqref{problem}, i.e., for the solution of standard nonlinear system of equations, the authors in \cite{cruz2003nonmonotone} introduced a method called SANE, which uses the residual $\pm F(X_k)$ as a search direction. Then the trial point, at each iteration, is computed by $X_{k+1}= X_k \pm \tau_k F(X_k)$, where $\tau_k$ is a spectral coefficient based on the Barzilai--Borwein step--size \cite{barzilai1988two,raydan1993barzilai}. This iterative process uses precisely the functional $F(\cdot)$, in order to define the search direction. It's feature of been a derivative--free procedure, is highly attractive, lowering the storage requirements and the computational cost per iteration.\\

Motivated by the Riemannian gradient and SANE methods, in this paper, we introduce RSANE, which is a generalization of SANE to tackle the numerical solution of nonlinear equations on Riemannian manifolds. In particular, we modified the update formula of SANE by incorporating a retraction, in order to guarantee that each point $X_k$ belongs to the desired manifold. By following the ideas of the Riemannian Barzilai--Borwein method developed by Iannazzo et.al in \cite{iannazzo2018riemannian}, we propose an extension of the spectral parameter $\tau_k$ to the case of Riemannian manifolds, using mappings so--called \emph{scaled vector transport}. In addition, we present the convergence analysis of the proposed method obtained under the Zhang--Hager globalization strategy \cite{zhang2004nonmonotone}. Finally, some numerical experiments are reported to illustrate the efficiency and effectiveness of our proposal.\\

The rest of this paper is organized as follow. To do this article self--contained, we briefly review, in Section \ref{sec:2}, some concepts and tools
from Riemannian geometry that it can be founded in \cite{absil2009optimization}. In Section \ref{sec:3}, we present our proposed Riemannian spectral residual method (RSANE) for solving \eqref{problem}. Section \ref{sec:4} is devoted to the convergence analysis concerning our proposed method. In Section \ref{sec:5} numerical tests are carried out, in order to illustrate the good performance of our approach considering the computing of eigenspaces associated to given symmetric matrices using both simulated data and real data. Finally,  conclusions and perspectives are provided in Section \ref{sec:6}.

\section{Preliminaries on Riemannian Geometry}
\label{sec:2}
In this section, we briefly review some notions and tools of Riemannian geometry crucial for understanding this paper, by summarizing \cite{absil2009optimization}.\\

Let $\mathcal{M}$  be a Riemannian manifold with an associated Riemannian metric $\langle \cdot,\cdot \rangle$, and let $T_{X}\mathcal{M}$ be
its tangent vector space at a given point $X\in \mathcal{M}$. In addition, let $\mathcal{F}:\mathcal{M}\to \mathbb{R}$ a smooth scalar function defined on the Riemannian manifold $\mathcal{M}$, the Riemannian gradient of $\mathcal{F}(\cdot)$ at $X$, denoted by $\nabla_{\mathcal{M}}\mathcal{F}(X)$, is defined as the unique element of $T_{X}\mathcal{M}$ that verifies
$$ \langle \nabla_{\mathcal{M}}\mathcal{F}(X),\xi_{X} \rangle = \mathcal{DF}(X)[\xi_X], \quad \forall \xi_{X}\in T_{X}\mathcal{M},  $$
where $\mathcal{DF}(X)[\xi_X]$ is the function that takes any point $X\in\mathcal{M}$ to the directional derivative of $\mathcal{F}(\cdot)$ in the direction $\xi_X$, evaluated at $X\in\mathcal{M}$. In the particular case that $\mathcal{M}$ is a Riemannian submanifold of an Euclidean space $\mathcal{E}$, we have an explicit evaluation of the gradient: let $\mathcal{\bar{F}}: \mathcal{E} \to \mathbb{R} $ a smooth function defined on $\mathcal{E}$ and let $\mathcal{F}:\mathcal{M}\subset\mathcal{E}\to \mathbb{R}$, then the Riemannian gradient of $\mathcal{F}(\cdot)$ evaluated at $X\in\mathcal{M}$ is equal to the orthogonal projection of the standard gradient of $\mathcal{\bar{F}}(\cdot)$ onto $T_{X}\mathcal{M}$, that is,
\begin{equation}
  \nabla_{\mathcal{M}}\mathcal{F}(X) = P_{X}[\nabla \mathcal{\bar{F}}(X)].
\end{equation}
This result provides us an important tool to compute the Riemannian gradient, which will be useful in the experiments section.\\

Another fundamental concept for this work is \emph{retraction}. This can be seen as a smooth function that pragmatically approximates the notion of geodesics \cite{edelman1998geometry}. Now we present its rigorous definition.
\begin{definition}[\cite{absil2009optimization}]
\label{def1}
A retraction on a manifold $\mathcal{M}$ is a smooth
mapping $R[\cdot]$ from the tangent bundle $\mathcal{TM}$ onto $\mathcal{M}$ with the following properties.
Let $R_X[\cdot]$ denote the restriction of $R[\cdot]$ to $T_X\mathcal{M}$.
\begin{enumerate}
  \item $R_X[0_X] = X$, where $0_X$ denotes the zero element of $T_X\mathcal{M}$
  \item With the canonical identification, $T_{0_X}T_{X}\mathcal{M}\simeq T_X\mathcal{M}$, $R_X[\cdot]$ satisfies
  $$ \mathcal{D}R_X[0_X] = \textrm{id}_{T_X\mathcal{M}},$$
  where $\textrm{id}_{T_X\mathcal{M}}$ denotes the identity mapping on $T_X\mathcal{M}$.

\end{enumerate}
\end{definition}
The second condition in Definition \ref{def1} is known as \emph{local rigidity condition}.\\

The concept of vector transport, which appears in \cite{absil2009optimization}, provides us a tool to perform operations between two or more vectors that belong to different tangent spaces of $\mathcal{M}$, and can be seen as a relaxation of the purely Riemannian concept of \emph{parallel transport} \cite{edelman1998geometry}.
\begin{definition}[\cite{absil2009optimization}]
\label{def2}
A vector transport $\mathcal{T}[\cdot]$ on a manifold $\mathcal{M}$ is a smooth mapping
$$ \mathcal{T}: \mathcal{TM}\oplus \mathcal{TM} \rightarrow  \mathcal{TM}: (\eta,\xi) \mapsto \mathcal{T}_{\eta}[\xi]\in \mathcal{TM}, $$
satisfying the following properties for all $X\in\mathcal{M}$ where $\oplus$ denote the Whitney sum, that is,
$$ \mathcal{TM}\oplus \mathcal{TM} = \{(\eta,\xi):\eta,\xi\in T_X\mathcal{M},X\in\mathcal{M}\}. $$
\begin{enumerate}
  \item There exists a retraction $R[\cdot]$, called the retraction associated with $\mathcal{T}[\cdot]$, such that
  $$ \pi(\mathcal{T}_{\eta}[\xi]) = R_X(\eta), \quad \eta,\xi\in T_X\mathcal{M}, $$
  where $\pi(\mathcal{T}_{\eta}[\xi])$ denotes the foot of the tangent vector $\mathcal{T}_{\eta}[\xi]$.

  \item $\mathcal{T}_{0_X}[\xi] = \xi$ for all $\xi\in T_{X}\mathcal{M}$.
  \item $\mathcal{T}_{\eta}[a\xi + b\zeta] = a\mathcal{T}_{\eta}[\xi] + b\mathcal{T}_{\eta}[\zeta]$, for all $a,b\in\mathbb{R}$ and $\eta,\xi,\zeta\in T_{X}\mathcal{M}$.
\end{enumerate}
\end{definition}

Next, the concept of isometry \cite{zhu2017riemannian} is established, which is a property satisfied by some vector transports.
\begin{definition}[\cite{zhu2017riemannian}]
\label{def3}
A vector transport $\mathcal{T}[\cdot]$ on a manifold $\mathcal{M}$ is called isometric if it satisfies
\begin{equation}
  \langle \mathcal{T}_{\eta}[\xi],\mathcal{T}_{\eta}[\xi] \rangle = \langle \xi,\xi \rangle, \label{IsometricT}
\end{equation}
for all $\eta,\xi\in T_{X}\mathcal{M}$, where $R_{X}[\cdot]$ is the retraction associated with $\mathcal{T}[\cdot]$.
\end{definition}

\section{Spectral Approach for Tangent Vector Field on Riemannian Manifolds}
\label{sec:3}
In this section, we shall establish our proposal RSANE. An intuitive way to solve \eqref{problem} is to promote the reduction of the residual $||F(\cdot)||$, which we can achieve by solving the following auxiliar manifold constrained optimization problem
\begin{equation}
  \min_{X\in\mathcal{M}} \mathcal{F}(X) = \frac{1}{2}||F(X)||^{2}. \label{problem2}
\end{equation}
We can deal with this optimization model using some Riemannian optimization method. Nevertheless, we are interested in directly solving the Riemannian nonlinear equation \eqref{problem}. For this purpose, we consider the following iterative method, based on the SANE method,
\begin{equation}
  X_{k+1} = X_k - \tau_k F(X_k). \label{eq1}
\end{equation}
Firstly, the vector $-F(X_k)$ is not necessarily a descent direction for the merit function $\mathcal{F}(\cdot)$ and secondly, that $X_{k+1}$ does not necessarily belong to the manifold $\mathcal{M}$. We can overcome the first one, by modifying this vector with the sign of $\langle  \nabla_{\mathcal {M}}\mathcal{F}(X_k),F(X_k)\rangle$, following the same idea used in SANE method, in order to force $\pm F(X_k)$ satisfies the descent condition. Observe that the Riemannian gradient method of $\mathcal{F}(\cdot)$ can be computed by
\begin{equation}
\nabla_{\mathcal{M}}\mathcal{F}(X) = (JF(X))^{*}[F(X)], \quad \forall X\in\mathcal{M}, \label{eq2}
\end{equation}
that is, to compute the Riemannian gradient of $\mathcal{F}(\cdot)$ at $X$, we need to calculate the adjoint of the Jacobian of $F(\cdot)$ evaluated at $X$.\\

On the other hand, the second disadvantage can be easily remedied by incorporating a retraction, similarly to the scheme \eqref{line_search}. Keeping in mid all theses considerations, we now propose our Riemannian spectral residual method, which computes the iterates recursively by
\begin{equation}
X_{k+1} = R_{X_{k}}[\tau_k Z_k], \label{eq3}
\end{equation}
where $\tau_k>0$ represents the step--size, $R_{X}[\cdot]$ is a retraction and the search direction is determined by
\begin{equation}
Z_{k} = -s_{\theta}(\langle \nabla_{\mathcal{M}}\mathcal{F}(X_k),F(X_k) \rangle) F(X_k), \label{eq4}
\end{equation}
where $s_{\theta}(\cdot):\mathbb{R}-\{0\}\to \{-1,1\}$ is defined by $s(x) = \frac{x}{|x|}$. Observe that $s(\cdot)$ is a continuous function for all $x\neq 0$, a crucial property to our convergence analysis.\\

\subsection{A Nonmonotone Line Search with a Riemannian Spectral Parameter}
In the scenario of the solution of nonlinear systems of equations over $\mathbb{R}^{n}$, the SANE method uses a spectral parameter $\tau_k$ inspired by the Barzilai--Borwein step--size, originally introduced in \cite{barzilai1988two} to speed up the convergence of gradient--type methods, in the context of optimization. SANE computes this spectral parameter as follow
\begin{equation}
\tau_{k+1}^{BB} = \textrm{sgn}_k\frac{S_{k}^{\top}S_{k}}{S_{k}^{\top}Y_{k}}, \label{eq5}
\end{equation}
where $S_{k} = X_{k+1} - X_{k}$, $Y_{k} = F(X_{k+1}) - F(X_k)$, $\textrm{sgn}_k = s( F(X_k)^{\top}J(X_k)F(X_k) )$ and $J(\cdot)$ denotes the Jacobian matrix of the vector--valued function $F:\mathbb{R}^{n}\to \mathbb{R}^{n}$.\\

In the framework of Riemannian manifolds, the vectors $F(X_{k+1})\in T_{X_{k+1}}\mathcal{M}$ and $F(X_k)\in T_{X_{k}}\mathcal{M}$ lie in different tangent spaces, then the difference between these vectors may not be well defined (this is only well defined over linear manifolds). The same drawback occurs with the difference between the points $X_{k+1}$ and $X_k$. Therefore, we cannot directly use the parameter \eqref{eq5} to address the numerical solution of \eqref{problem}. In \cite{iannazzo2018riemannian} Iannazzo et. al. was extended the Barzilai--Borwein step--sizes in the context of optimization on Riemannian manifolds, through the use of a vector transport (see Definition \ref{def2}). This strategy transports the calculated directions to the correct tangent space, providing a way to overcome the drawback. Guided by the descriptions contained in \cite{iannazzo2018riemannian}, we propose the following generalization of the spectral parameter \eqref{eq5},
\begin{equation}
\tau_{k+1}^{RBB1} = s(\sigma_k)\frac{\langle \hat{S}_{k},\hat{S}_{k}\rangle}{\langle \hat{S}_{k},\hat{Y}_{k}\rangle}, \label{eq6}
\end{equation}
where $\sigma_k = \langle \nabla_{\mathcal{M}}\mathcal{F}(X_k),F(X_k) \rangle$,
\begin{equation}
\hat{S}_{k} = \mathcal{T}_{\tau_kZ_k}[\tau_k Z_k] = -\tau_{k}s(\sigma_k) \mathcal{T}_{\tau_kZ_k}[F(X_k)] \label{eq7}
\end{equation}
and
\begin{equation}
\hat{Y}_{k}  = F(X_{k+1}) - \mathcal{T}_{\tau_kZ_k}[F(X_k)] = F(X_{k+1}) + \frac{1}{\tau_k s(\sigma_k)} \hat{S}_k, \label{eq8}
\end{equation}
where $\mathcal{T}[\cdot]$ is any vector transport satisfying the Ring--Wirth non--expansive condition,
\begin{equation}
||\mathcal{T}_{\eta_X}[\xi_X]|| \leq ||\xi_X||, \quad \forall \xi_X,\eta_X\in T_X\mathcal{M}. \label{eq9}
\end{equation}

Another alternative for the spectral parameter is
\begin{equation}
\tau_{k+1}^{RBB2} = s(\sigma_k)\frac{\langle \hat{S}_{k},\hat{Y}_{k}\rangle}{\langle \hat{Y}_{k},\hat{Y}_{k}\rangle}. \label{eq6a}
\end{equation}
In order to take advantage of both spectral parameters $\tau_{k+1}^{RBB1}$ and $\tau_{k+1}^{RBB2}$, we adopt the following adaptive strategy
\begin{equation}
\tau^{RBB}_{k+1} = \left\{ \begin{array}{ll}
\tau_{k+1}^{RBB1} & \textrm{for even k};\\
\tau_{k+1}^{RBB2} & \textrm{for odd k}.
\end{array} \right. \label{RBBstep}
\end{equation}

Note that it is always possible to define a transporter (a function that sends vectors from a tangent space to another tangent space) that satisfies the condition \eqref{eq9}, by scaling,
\begin{equation}\label{eq10}
\mathcal{T}_{\eta_X}^{\textrm{scaled}}[\xi_X] = \left\{ \begin{array}{ll}
\mathcal{T}_{\eta}[\xi] & \textrm{if } ||\mathcal{T}_{\eta_X}[\xi_X]|| \leq ||\xi_X||;\\
\frac{||\xi||}{||\mathcal{T}_{\eta}[\xi]||}\mathcal{T}_{\eta}[\xi] & \textrm{otherwise}.
\end{array} \right.
\end{equation}
This function, introduced by Sato and Iwai in \cite{sato2015new}, is referred as \emph{scaled vector transport}. Observe that \eqref{eq10} is not necessarily a vector transport. However for the extension of the spectral parameter to the setting of Riemannian manifolds, it is not strictly mandatory. In fact, it is enough having a non--expansive transporter available. Therefore, we will use the scaled vector transport \eqref{eq10}, in the construction of the vectors $\hat{S}_{k}$ and $\hat{Y}_{k}$ in equations \eqref{eq7}--\eqref{eq8}.\\

Since the spectral parameter $\tau_{k}^{RBB}$ does not necessarily reduces the value of the merit function $\mathcal{F}(\cdot)$ at each iteration, the convergence result could be invalid. We can overcome this drawback by incorporating a globalization strategy,  which guarantees convergence by regulating the step--size $\tau_k$ only occasionally (see \cite{cruz2003nonmonotone,raydan1997barzilai,zhang2004nonmonotone}). In the seminal paper \cite{cruz2003nonmonotone}, the authors consider the globalization technique proposed by Grippo et.al. in \cite{grippo1986nonmonotone}, in the definition of SANE method.

We could define our Riemannian generalization of SANE incorporating this non--monotone technique, and so analyze the convergence following the ideas described in \cite{cruz2003nonmonotone,iannazzo2018riemannian}. Instead of that, in this work, to define RSANE we adopt a more elegant globalization strategy proposed by Zhang and Hager in \cite{zhang2004nonmonotone}. Specifically, we compute $\tau_k = \delta^{h} \tau_k^{RBB}$ where $h\in\mathbb{N}$ is the smallest integer satisfying
\begin{equation}
\mathcal{F}(R_{X_k}[\tau_k Z_k]) \leq C_k + \rho_1\tau_k \langle \nabla_{\mathcal{M}}\mathcal{F}(X_{k}), Z_k \rangle, \label{eq11}
\end{equation}
where each value $C_{k+1}$ is given by a convex combination of $\mathcal{F}(X_{k+1})$ and the previous $C_k$ as
$$C_{k+1} = \frac{\eta Q_kC_k + \mathcal{F}(X_{k+1})}{Q_{k+1}},$$
for $Q_{k+1} = \eta Q_k + 1$, starting at $Q_0 = 1$ and $C_0 = \mathcal{F}(X_0)$. In the sequel our generalization RSANE will be described in detail.

\begin{algorithm}[H]
\begin{algorithmic}[1]
\REQUIRE Let $X_{0}\in \mathcal{M}$ be the initial guess, $\tau>0$, $0<\tau_m \leq \tau_M < \infty$, $\eta\in[0,1)$, $\rho_1,\epsilon,\epsilon_1,\delta \in (0,1)$, $Q_0 = 1$, $C_0 = \mathcal{F}(X_0)$,  $k=0$.
\WHILE{ $|| F(X_k) || > \epsilon$ }
\STATE $\sigma_k = \langle \nabla_{\mathcal{M}}\mathcal{F}(X_k),F(X_k) \rangle$,
\IF{ $| \sigma_k | < \epsilon_1 ||F(X_k)||^2$  }
\STATE stop the process.
\ENDIF
\STATE $Z_k = - s(\sigma_k) F(X_k)$,
\WHILE { $\mathcal{F}(R_{X_{k}}[\tau Z_k]) > C_k - \rho_1\epsilon_1\tau ||F(X_{k})||^{2}$}
\STATE $\tau \leftarrow \delta\tau$,
\ENDWHILE
\STATE $\tau_k = \tau$,
\STATE $X_{k+1} = R_{X_{k}}[\tau_k Z_k]$,
\STATE $Q_{k+1} = \eta Q_k + 1$ and $ C_{k+1} = (\eta Q_kC_k + \mathcal{F}(X_{k+1}) )/Q_{k+1}$,
\STATE $\hat{S}_k = -\tau_k s(\sigma_k) \mathcal{T}_{\tau_k Z_k}[F(X_k)]$  and  $\hat{Y}_k = F(X_{k-1}) + \frac{1}{\tau_k s(\sigma_k)} \hat{S}_k$,
\STATE $\tau_{k+1}^{RBB} = s(\sigma_k)\frac{\langle \hat{S}_{k},\hat{S}_{k}\rangle}{\langle \hat{S}_{k},\hat{Y}_{k}\rangle}$,
\STATE $ \tau = \max(\min( \tau_{k+1}^{RBB}, \tau_M ),\tau_m )$.
\STATE $k \leftarrow k+1$.
\ENDWHILE
\end{algorithmic}
\caption{Spectral Residual Method for tangent vector field (RSANE)}\label{Alg1}
\end{algorithm}

\begin{remark}
In Algorithm \ref{Alg1}, we replace the nonmonotone condition \eqref{eq11} by
\begin{equation}
\mathcal{F}(R_{X_k}[\tau Z_k]) \leq C_k - \rho_1\epsilon_1\tau ||F(X_k)||^2. \label{eq11b}
\end{equation}
We remark that with this relaxed condition, Algorithm \ref{Alg1} is well defined. In fact, if at iteration $k$ the procedure does not stop at Step 4, then $Z_k$ is a descent direction (see Lemma \ref{lemma2}), and for all $\rho_1\in(0,1)$ there exists $t>0$ such that the non--monotone Zhang--Hager condition \eqref{eq11} holds by continuity, for $\tau>0$ sufficiently small(a proof of this fact appears in \cite{zhang2004nonmonotone}). In addition, it follows form Step 4, \eqref{eq11} and Lemma \ref{lemma2} that
\begin{eqnarray}
\mathcal{F}(R_{X_k}[\tau Z_k]) & \leq & C_k + \rho_1\tau \langle \nabla_{\mathcal{M}}\mathcal{F}(X_{k}), Z_k \rangle \nonumber \\
                                 & \leq & C_k - \rho_1\epsilon_1\tau ||F(X_k)||^2, \nonumber
\end{eqnarray}
which implies that the relaxed condition \eqref{eq11b} is also verified for all $\tau>0$ that satisfy \eqref{eq11}.
\end{remark}

\begin{remark}
The bottleneck of Algorithm \ref{Alg1} appears in step 2, since to calculate $\sigma_k$, we must compute the Riemannian gradient of $\mathcal{F}(\cdot)$, which implies evaluating the Jacobian of $F(\cdot)$ (see equation \eqref{eq2}). However, given a retraction $R_{X}[\cdot]$, $\sigma_k$ can be approximated using finite differences as follow
\begin{eqnarray}
\sigma_k & = & \langle \nabla_{\mathcal{M}}\mathcal{F}(X_k),F(X_k) \rangle  \nonumber \\
         & = & \mathcal{DF}(X_k)[F(X_k)] \nonumber \\
         & = & \lim_{t\to 0} \frac{\mathcal{F}(R_{X_k}[tF(X_k)]) - \mathcal{F}(X_k)}{t} \nonumber \\
         & \approx &  \frac{\mathcal{F}(R_{X_k}[hF(X_k)]) - \mathcal{F}(X_k)}{h},    \label{aprox_deriv}
\end{eqnarray}
where $0< h \ll 1$ is a small real number. The fact that this approximation does not need the explicit knowledge of the Jacobian operator is useful for large--scale problems.
\end{remark}

The following lemma establishes that in most cases $\tau_k^{RBB1}$ is positive.
\begin{lemma}\label{lemma1}
Let $\tau_{k+1}^{RBB1}$ be computed as in \eqref{eq6}, then $\tau_{k+1}^{RBB1} > 0$ when one of the following
cases holds
\begin{enumerate}
  \item $\langle \mathcal{T}_{\tau_kZ_k}[F(X_k)], F(X_{k+1}) \rangle < 0$;
  \item $\langle \mathcal{T}_{\tau_kZ_k}[F(X_k)], F(X_{k+1}) \rangle > 0 $ \textrm{and} $||F(X_{k+1})|| < ||\mathcal{T}_{\tau_kZ_k}[F(X_k)]||$.
\end{enumerate}
\end{lemma}
\begin{proof}
Since $\tau_{k+1}^{RBB1} = s(\sigma_k)\frac{\langle \hat{S}_{k},\hat{S}_{k}\rangle}{\langle \hat{S}_{k},\hat{Y}_{k}\rangle} = \frac{-1}{\tau_k}\frac{||\hat{S}_{k}||^2}{\langle \mathcal{T}_{\tau_kZ_k}[F(X_k)],\hat{Y}_{k}\rangle}$ then $\verb"sign"(\tau_{k+1}^{RBB1}) = \verb"sign"(-\langle \mathcal{T}_{\tau_kZ_k}[F(X_k)],\hat{Y}_{k}\rangle)$. Suppose that (a) holds, then
$$\langle \mathcal{T}_{\tau_kZ_k}[F(X_k)],\hat{Y}_{k}\rangle = \langle \mathcal{T}_{\tau_kZ_k}[F(X_k)],F(X_{k+1})\rangle - ||\mathcal{T}_{\tau_kZ_k}[F(X_k)]||^2 < 0,$$
which implies that $\tau_{k+1}^{RBB} > 0$.\\

On the other hand, if (b) holds, from the Cauchy--Schwarz inequality, we find that
$$ 0 < \langle \mathcal{T}_{\tau_kZ_k}[F(X_k)], F(X_{k+1}) \rangle \leq ||\mathcal{T}_{\tau_kZ_k}[F(X_k)]||\, ||F(X_{k+1})|| < ||\mathcal{T}_{\tau_kZ_k}[F(X_k)]||^{2}, $$
and so
$$ \langle \mathcal{T}_{\tau_kZ_k}[F(X_k)],\hat{Y}_{k}\rangle = \langle \mathcal{T}_{\tau_kZ_k}[F(X_k)],F(X_{k+1})\rangle - ||\mathcal{T}_{\tau_kZ_k}[F(X_k)]||^2 < 0, $$
and hence $\tau_{k+1}^{RBB1} > 0$, which proves the lemma.
\end{proof}
The same theoretical result is verified for spectral parameter $\tau_{k+1}^{RBB2}$, and therefore it is also valid for the adaptive parameter $\tau_{k+1}^{RBB}$.\\

Notice that if the transporter $\mathcal{T}[\cdot]$ is isometric (see Definition \ref{def3}), then the second condition of item (b) in Lemma \ref{lemma1} is reduced to $||F(X_{k+1})|| < ||F({X_k})||$. In addition, observe that if we set $\eta = 0$ in Algorithm \ref{Alg1}, then we have that $\mathcal{F}(X_{k+1})< \mathcal{F}(X_k)$ for all $k\in\mathbb{N}$. Therefore this condition would always be verified under these choices of $\eta$ and $\mathcal{T}[\cdot]$.

\begin{lemma}\label{lemma2}
Assume that Algorithm \ref{Alg1} does not terminate. Let $\{Z_k\}$ be an infinite sequence of tangent direction generated by Algorithm \ref{Alg1}. Then $Z_{k}$ is a descent direction for the merit function $\mathcal{F}(\cdot)$ at $X_k$, for all $k\geq 0$.
\end{lemma}
\begin{proof}
From Step 6 in Algorithm \ref{Alg1} we have $Z_k = - s(\sigma_k)F(X_k)$, where $\sigma_k = \langle \nabla_{\mathcal{M}}\mathcal{F}(X_k), F(X_k) \rangle$. Since Algorithm \ref{Alg1} does not terminate, form Step 3 we obtain $|\sigma_k| \geq \epsilon_1 ||F(X_k)||^{2} > 0$. Now, observe that
\begin{eqnarray}
\langle \nabla_{\mathcal{M}}\mathcal{F}(X_k), Z_k \rangle & = & - s(\sigma_k)\langle \nabla_{\mathcal{M}}\mathcal{F}(X_k), F(X_k) \rangle  \nonumber \\
                                                          & = & - s(\sigma_k)\sigma_k \nonumber \\
                                                          & = & - \frac{\sigma_k^2}{|\sigma_k|} \nonumber \\
                                                          & = & - |\sigma_k| \nonumber \\
                                                          & < & 0. \nonumber
\end{eqnarray}
Therefore, we conclude that $Z_k$ is a descent direction for $\mathcal{F}(\cdot)$ at $X_k$, for all $k\geq 0$.
\end{proof}

\section{Convergence Analysis}
\label{sec:4}
In this section, we analyse the global convergence for our Algorithm \ref{Alg1} under mild assumptions. Our analysis consists on a generalization of the global convergence of line--search methods for unconstrained optimization, presented in \cite{zhang2004nonmonotone} and an adaptation of Theorem 4.3.1 in \cite{absil2009optimization}.\\

The following lemma establishes that ${C_k}$ is bounded below by the sequence $\{\mathcal{X_k}\}$.
\begin{lemma}\label{lemma3}
 Let $\{X_k\}$ be an infinite sequence generated by Algorithm \ref{Alg1}. Then for the iterates generated by Algorithm \ref{Alg1}.
 \begin{equation}
 \mathcal{F}(X_k) \leq  C_k, \quad \forall k. \label{CA_eq1}
 \end{equation}
\end{lemma}
\begin{proof}
Firstly, we define $\psi_k:\mathbb{R}\to \mathbb{R}$ by
\begin{equation}
\psi(\alpha) = \frac{\alpha C_{k-1} + \mathcal{F}(X_k)}{\alpha + 1}, \nonumber
\end{equation}
observe that the derivative of $\psi(\alpha)$ is
\begin{equation}
\dot{\psi}(\alpha) =  \frac{C_{k-1} - \mathcal{F}(X_k)}{(\alpha + 1)^2}. \nonumber
\end{equation}
it follows from the non--monotone condition \eqref{eq11b} that
\begin{equation}
\mathcal{F}(X_k) \leq C_{k-1} - \rho_1\epsilon_1\tau ||F(X_k)||^{2} < C_{k-1},
\end{equation}
which implies that $\dot{\psi}(\alpha)\geq 0$ for all $\alpha\geq 0$. Hence, the function $\psi(\cdot)$ is nondecreasing, and $\mathcal{F}(X_k) = \psi(0) \leq \psi(\alpha)$ for all $\alpha \geq 0$. Then, taking $\bar{\alpha} = \eta Q_{k-1}$ we obtain
\begin{equation}
\mathcal{F}(X_k) = \psi(0) \leq \psi(\bar{\alpha}) = C_k,
\end{equation}
which completes the proof.
\end{proof}

In order to prove the global convergence of our proposed algorithm, we need the following asymptotic property.
\begin{lemma}\label{lemma4}
 Any infinite sequence $\{X_k\}$ generated by Algorithm \ref{Alg1} verifies the following property
 \begin{equation}
 \lim_{k\to\infty} \tau_k ||F(X_k)||^2 = 0.
 \end{equation}
\end{lemma}
\begin{proof}
By the construction of Algorithm \ref{Alg1} and using Lemma \ref{lemma2}, we have
\begin{equation}
C_{k+1} = \frac{\eta Q_kC_k + \mathcal{F}(X_{k+1})}{Q_{k+1}} < \frac{(\eta Q_{k} + 1)C_k}{Q_{k+1}} = C_k. \label{inqC}
\end{equation}
Hence, $\{C_k\}$ is monotonically decreasing and bounded below by zero, therefore it converges to some limit $C^{*}\geq 0$. It follows from Step 7 and Step 12 in Algorithm \ref{Alg1} that
\begin{equation}
\sum_{k=0}^{\infty} \frac{\rho_1\epsilon_1\tau_k ||F(X_k)||^2}{Q_{k+1}} \leq \sum_{k=0}^{\infty}C_k - C_{k+1} = C_0 - C_{*} < \infty.
\end{equation}
Merging this result with the fact that $Q_{k+1} = 1 + \eta Q_k = 1 + \eta + \eta^2Q_{k-1} = \cdots = \sum_{i=0}^{k} \eta^{i} < (1-\eta)^{-1}$, we have
\begin{equation}
\lim_{k\to \infty} \tau_k||F(X_k)||^2   = 0,
\end{equation}
which proves the lemma.
\end{proof}

The theorem below establishes a global convergence property concerning Algorithm \ref{Alg1}. The proof can be seen as a modification to that of Theorem 3.4 in \cite{yao2020riemannian}, and to that of Theorem 4.1 in \cite{oviedoimplicit}.
\begin{theorem}\label{theorem1}
Algorithm \ref{Alg1} either terminates at a finite iteration $j\in\mathbb{N}$
where $|\langle \nabla_{\mathcal{M}}\mathcal{F}(X_j), F(X_j) \rangle | < \epsilon ||F(X_j)||^{2}$, or  it generates an infinite sequence $\{X_k\}$ such that
$$\lim_{k\to \infty} \emph{inf } ||F(X_k)|| = 0. $$
\end{theorem}
\begin{proof}
Let us assume that Algorithm \ref{Alg1} does not terminate, and let $X_{*}$ be any accumulation point of the sequence $\{X_k\}$. We may assume that $\lim_{k\to \infty} X_k = X_{*}$, taking a subsequence if necessary. By contradiction, suppose that there exists $\epsilon_0 > 0$ such that
\begin{equation}
||F(X_k)||^2 > \epsilon_0, \quad \forall k\geq 0. \label{eq0_teo}
\end{equation}
In view of \eqref{eq0_teo} and Lemma \ref{lemma4} we have
\begin{equation}
  \lim_{k\to\infty} \textrm{inf } \tau_k = 0. \label{eq1_teo}
\end{equation}
Firstly, we define the curve $Y_k(\tau) = R_{X_k}[-\tau Z_k]$ for all $k\in \mathbb{N}$, which is smooth due to the differentiability of the retraction $R_{X}[\cdot]$. Since the parameter $\tau_k>0$ is chosen by carrying out a backtracking process, then $\tau_k = \delta^{m_k}\tau_k^{RBB}$, for all $k$ greater than some $\bar{k}$, where  $m_k\in\mathbb{N}$ is the smallest positive integer number such that the relaxed nonmonotone condition \eqref{eq11b} is fulfilled. Thus, the scalar $\bar{\tau} = \frac{\tau_k}{\delta}$ violates the condition \eqref{eq11b}, i.e., it holds
\begin{equation}
-\rho_1\epsilon_1\bar{\tau} ||F(X_k)||^{2} < \mathcal{F}(Y_k(\bar{\tau})) - C_k \leq \mathcal{F}(Y_k(\bar{\tau})) - \mathcal{F}(X_k), \quad \forall k\geq \bar{k}, \label{eq2_teo}
\end{equation}
where the last inequality is obtained using Lemma \ref{lemma3}.\\

Let us set $\psi_k(\tau) := \mathcal{F} \circ Y_k(\tau)$, then \eqref{eq2_teo} is equivalent to
\begin{equation}
- \frac{\psi_k(\bar{\tau}) - \psi_k(0)}{\bar{\tau}-0} < \rho_1\epsilon_1 ||F(X_k)||^{2}, \quad \forall k\geq \bar{k}. \label{eq3_teo}
\end{equation}
It follows from the mean value theorem, that there exists $t\in(0,\bar{\tau})$ such that $-\dot{\psi}_k(t) < \rho_1\epsilon_1||F(X_k)||^{2}$ for all $k\geq \bar{k}$, or equivalently
\begin{equation}
-\langle \nabla_{\mathcal{M}}\mathcal{F}(Y_k(t)), \dot{Y}_k(t) \rangle < \rho_1\epsilon_1||F(X_k)||^{2}, \quad \forall k\geq \bar{k}. \label{eq4_teo}
\end{equation}

In view of the continuity of functions $s(\cdot)$, $\nabla_{\mathcal{M}}\mathcal{F}(\cdot)$, $F(\cdot)$, $Y_k(\cdot)$, the smoothness and local rigidity condition of the retraction $R_X[\cdot]$, and taking limit in \eqref{eq4_teo}, we arrive at
\begin{equation}
-\langle \nabla_{\mathcal{M}}\mathcal{F}(X_{*}), Z_{*} \rangle \leq \rho_1\epsilon_1||F(X_{*})||^{2}, \nonumber
\end{equation}
or equivalently,
\begin{equation}
|\sigma_{*}| \leq \rho_1\epsilon_1||F(X_{*})||^{2}, \label{eq4b_teo}
\end{equation}
where $|\sigma_{*}| = |\langle \nabla_{\mathcal{M}}\mathcal{F}(X_{*}), F(X_{*}) \rangle|$. Since Algorithm \ref{Alg1} does not terminate, form Step 3 we have
\begin{equation}
|\sigma_{k}|\geq \epsilon_1 ||F(X_k)||^{2}. \label{eq4c_teo}
\end{equation}
Applying limits in \eqref{eq4c_teo} we find that $|\sigma_{*}| \geq \epsilon_1||F(X_{*})||^{2}$. Merging this last result with \eqref{eq4b_teo}, we arrive at
\begin{equation}
||F(X_{*})||^{2}\leq \rho_1||F(X_{*})||^{2}. \nonumber
\end{equation}
Since $0<\rho_1 < 1$ then we have $||F(X_{*})|| = 0$, this last result contradicts \eqref{eq0_teo}, which completes the proof.
\end{proof}

By Theorem \ref{theorem1}, we obtain the following theoretical consequence under compactness assumptions.
\begin{corollary}\label{coro1}
Let $\{X_k\}$ be an infinite sequence generated by Algorithm \ref{Alg1}. Suppose that the level set $\mathcal{L} = \{ X\in\mathcal{M}: \mathcal{F}(X) \leq \mathcal{F}(X_0) \}$ is compact (which holds in particular when the Riemannian manifold $\mathcal{M}$ is compact). Then
$$ \lim_{k\to \infty} ||F(X_k)|| = 0. $$
\end{corollary}
\begin{proof}
Let $\{X_k\}$ be an infinite sequence generated by Algorithm \ref{Alg1}. It follows from Lemma \ref{lemma3}, the strict inequality \eqref{inqC} and the construction of
Algorithm \ref{Alg1} that
$$ \mathcal{F}(X_{k+1}) \leq C_{k+1} < C_k < \cdots < C_1 < C_0 = \mathcal{F}(X_0), $$
which implies that $X_k\in \mathcal{L}$, for all $k\geq 0$.\\

Now, by contradiction let us suppose that there is a subsequence $\{X_k\}_{k\in\mathcal{K}}$ and $\epsilon_0>0$ such that
\begin{equation}\label{eq1_coro}
  ||F(X_k)|| \geq \epsilon_0,
\end{equation}
for any $k\in\mathcal{K}$. Since $X_k\in \mathcal{L}$, for all $k\geq 0$ and $L$ is a compact set, we have that $\{X_k\}_{k\in\mathcal{K}}$ must have an accumulation point $X_{*}\in \mathcal{L}$. Taking limit in \eqref{eq1_coro} and using the continuity of $||F(\cdot)||$, we obtain $||F(X_{*})||\geq \epsilon_0$, which contradicts Theorem \ref{theorem1}.
\end{proof}

\begin{remark}
The main drawback of Algorithm \ref{Alg1} is that it can prematurely terminate with a \emph{bad breakdown} ($| \sigma_k | < \epsilon_1 ||F(X_k)||^2$). One way to remedy this problem is to use $Z_k = -\nabla\mathcal{F}(X_k)$ as the search direction, if a \emph{bad breakdown} occurs at $k$--th iteration; and then in the next iteration, we can retry using the steps of Algorithm \ref{Alg1}. We may even use any tangent direction $Z_k\in T_{X_k}\mathcal{M}$ such that $\langle \nabla\mathcal{F}(X_k), Z_k\rangle < 0$, as long as a \emph{bad breakdown} happens, in order to overcome this difficulty.
\end{remark}

\section{Numerical Experiments}
\label{sec:5}
In order to give further insight into the RSANE method we present the results of some numerical experiments. We test our algorithm on some randomly generated gradient tangent vector fields on three different Riemannian manifold, involving the unit sphere, the Stiefel manifold and oblique manifold. All experiments have been performed
on a intel(R) CORE(TM) i7--4770, CPU 3.40 GHz with 1TB HD and 16GB RAM memory. The algorithm was coded in Matlab with double precision. The running times are always given in CPU seconds. For numerical comparisons, we consider the SANE method proposed in \cite{cruz2003nonmonotone}, and the recently published Riemannian Derivative--Free Conjugate gradient Polak--Ribi\'ere--Polyak method (CGPR) for the numerical solution of tangent vectors field \cite{yao2020riemannian}. The Matlab codes of our RSANE, SANE and CGPR are available in: \url{http://www.optimization-online.org/DB_HTML/2020/09/8028.html}

\section{Implementation details}
In our implementation, in addition to monitoring the residual norm $||F(X_k)||_F$, we also check the relative changes of the two consecutive iterates and their corresponding
residual values
\begin{equation}\label{relres}
\textrm{rel}_k^{X} := \frac{||X_{k+1} - X_{k}||_F}{||X_k||_F} \quad \textrm{and} \quad  \textrm{rel}_k^{F} := \frac{|\mathcal{F}(X_{k+1})-\mathcal{F}(X_k)|}{\mathcal{F}(X_k) + 1}.
\end{equation}
Here $||M||_F$ denotes the Frobenius norm of the matrix $M$. In the case when $M$ is a vector, this norm is reduced to the standard norm on $\mathbb{R}^{n}$. Although the residual $||X_{k+1} - X_{k}||_F$ is meaningless in the Riemannian context, for our numerical experiments, we will only consider Riemannian manifolds embedded in the Euclidean space $\mathbb{R}^{n\times p}$, and the residual $\textrm{rel}_k^{X}$ is well defined for these types of manifolds.\\

We let all the algorithms run up to $K$ iterations and stop them at iteration $k < K$ if $||F(X_k)||_F < \epsilon$, or $\textrm{rel}_{k}^{X} < \epsilon_{X}$ and $\textrm{rel}_{k}^{F} < \epsilon_{F}$, or
$$ \textrm{mean}( [\textrm{rel}_{k - \min{k,T} + 1}^{X},\ldots,\textrm{rel}_{k}^{X}] )\leq 10\epsilon_X \,\, \textrm{and} \,\, \textrm{mean}( [\textrm{rel}_{k - \min{k,T} + 1}^{F},\ldots,\textrm{rel}_{k}^{F}] )\leq 10\epsilon_F.  $$
Here, the defaults values of $\epsilon,\epsilon_X,\epsilon_F$ and $T$ are $1e$-5, $1e$-15, $1e$-15 and $5$, respectively. In addition, in Algorithm \ref{Alg1} we use $\eta = 0.6$, $\tau = 1$e-3 (the initial step--size $\tau$), $\tau_{m} = 1e$-10, $\tau_{m} = 1e$+10, $\delta = 0.2$, $\epsilon_1 = 1e-8$ and $\rho_1 = 1e$-4 as defaults values.

\subsection{Considered manifolds and their geometric tools}
In this subsection, we present three Riemannian manifolds that we will use for the numerical experiments in the remainder subsections, as well as some tools necessary for the algorithms, associated with each manifold, such as vectors transports and retractions.\\

Firstly we considere the \emph{unit sphere} given by
\begin{equation}
S^{n-1} = \{x\in\mathbb{R}^{n}: ||x||_2 = 1\}. \label{sphere}
\end{equation}
It is well--known that the tangent space of the unit sphere at $x\in S^{n-1}$ is given by $T_xS^{n-1} = \{ z\in\mathbb{R}^{n}: z^{\top}x = 0 \}$. Let $S^{n-1}$ be endowed with the inner product inherited from the classical inner product on $\mathbb{R}^{n}$,
$$ \langle \xi_x, \eta_x \rangle_x := \xi_x^{\top}\eta_X,  \quad \forall \xi_x,\eta_x\in T_xS^{n-1},\, x\in S^{n-1}. $$
Then, $S^{n-1}$ with this inner product defines an $n-1$ dimensional Riemannian sub--manifold of $\mathbb{R}^{n}$. The retraction $R_x[\cdot]$ on $S^{n-1}$ is chosen as in \cite{absil2009optimization},
$$ R_x[\xi_x] = \frac{x + \xi_x}{||x + \xi_x||_2}, $$
for all $\xi_x\in T_xS^{n-1}$ and $x\in S^{n-1}$. In addition, for this particular manifold, we consider the vector transport based on orthogonal projection
$$ \mathcal{T}_{\eta_x}[\xi_x] = \xi_x - R_x[\eta_x]R_x[\eta_x]^{\top}\xi_x. $$
Notice that this vector transport verifies the Ring--Wirth non-expansive condition \eqref{eq9}.\\

The second Riemannian manifold considered in this section is the \emph{Stiefel manifold}, which is defined as
\begin{equation}
St(n,p) = \{X\in\mathbb{R}^{n\times p}: X^{\top}X = I_p\}, \label{Stiefel}
\end{equation}
where $I_p\in\mathbb{R}^{p\times p}$ denotes the identity matrix of size $p$--by--$p$. By differentiating both sides of $X(t)^{\top}X(t) = I_p$, we obtain the tangent space of $St(n,p)$ at $X$, given by $T_{X}St(n,p) = \{ Z\in\mathbb{R}^{n\times p}: Z^{\top}X + X^{\top}Z = 0 \}$. Let $St(n,p)$ be endowed with induced Riemannian metric from $\mathbb{R}^{n\times p}$, i.e.,
\begin{equation}
\langle \xi_X, \eta_X \rangle_X := tr(\xi_X^{\top}\eta_X),  \quad \forall \xi_x,\eta_x\in T_XSt(n,p),\, X\in St(n,p).   \label{metric_stiefel}
\end{equation}
The pair $(St(n,p),\langle \cdot,\cdot\rangle)$, where $\langle \cdot,\cdot\rangle$ is the inner product given in \eqref{metric_stiefel}, forms a Riemannian sub--manifold of the Euclidean space $\mathbb{R}^{n\times p}$, and its dimension is equal to $np - \frac{1}{2}p(p+1)$, \cite{absil2009optimization}. For our numerical experiments concerning the Stiefel manifold we will use the retractions, introduced in \cite{absil2009optimization}, given by
\begin{equation}
R_X[\xi_X] = \verb"qf"(X+\xi_X),   \label{qrR_stiefel}
\end{equation}
and the retraction based on the matrix polar decomposition
\begin{equation}
R_X[\xi_X] = (X+\xi_X)\left((X+\xi_X)^{\top}(X+\xi_X)\right)^{-1/2},   \label{polarR_stiefel}
\end{equation}
for all $\xi_X\in T_{X}St(n,p)$ and $X\in St(n,p)$. In equation \eqref{qrR_stiefel}, $\verb"qf"(W)$ denotes the orthogonal factor $Q$ obtained form the QR--factorization of $W$, such that $W = QR,$ where $Q$ belongs $St(n,p)$ and $R\in \mathbb{R}^{p\times p}$ is the upper triangular matrix with strictly positive diagonal elements. Additionally, we consider the following vector transport
\begin{equation}
\mathcal{T}_{\eta_X}[\xi_X] = \xi_X - R_X[\eta_X]\verb"sym"(R_X[\eta_X]^{\top}\xi_X), \label{VT_stiefel}
\end{equation}
where $R_X[\cdot]$ is any of the two retractions defined in \eqref{qrR_stiefel}--\eqref{polarR_stiefel}, and $\verb"sym"(W) = \frac{1}{2}(W^{\top} + W)$ is the function that assigns to the matrix $W$ its symmetric part. This vector transport is inspired by the orthogonal projection on the tangent space of the Stiefel manifold. Moreover, the function \eqref{VT_stiefel} verifies the Ring--Wirth non-expansive condition \eqref{eq9}.\\

Our third example, the \emph{oblique manifold}, is defied as
\begin{equation}
\mathcal{OB}(n,p) = \{X\in\mathbb{R}^{n\times p}: \verb"ddiag"(X^{\top}X) = I_p\}, \label{Oblique}
\end{equation}
where $\verb"ddiag"(W)$ denotes the matrix $W$ with all its off--diagonal entries assigned to zero. The tangent space associated to $\mathcal{OB}(n,p)$ at $X$ is given by $T_{X}\mathcal{OB}(n,p) = \{ Z\in\mathbb{R}^{n\times p} : \verb"ddiag"(X^{\top}Z) = 0\}$. Again, if we endow $OB(n,p)$ with the inner product inherited from the standard inner product on $\mathbb{R}^{n\times p}$, given by \eqref{metric_stiefel}, then the $\mathcal{OB}(n,p)$ becomes an embedded Riemannian manifold on $\mathbb{R}^{n\times p}$. For this particular manifold, we consider the retraction, which appears in \cite{absil2006joint}, defined by
\begin{equation}
R_X[\xi_X] = (X+\xi_X)\verb"ddiag"\left((X+\xi_X)^{\top}(X+\xi_X)\right)^{-1/2},   \label{polarR_oblique}
\end{equation}
for all $\xi_X\in T_{X}\mathcal{OB}(n,p)$ and $X\in \mathcal{OB}(n,p)$. We will use another vector transport based on the orthogonal projection operator on $T_{X}\mathcal{OB}(n,p)$,
\begin{equation}
\mathcal{T}_{\eta_X}[\xi_X] = \xi_X - R_X[\eta_X]\verb"ddiag"(R_X[\eta_X]^{\top}\xi_X), \label{VT_oblique}
\end{equation}
for all $\xi_X,\eta_X\in T_{X}\mathcal{OB}(n,p)$ and $X\in \mathcal{OB}(n,p)$.

\subsection{Eigenvalues computation on the sphere}
For the first test problem, we consider the standard eigenvalue problem
\begin{equation}
Ax = \lambda x, \label{eig_problem1}
\end{equation}
where $A\in\mathbb{R}^{n\times n}$ is a given symmetric matrix. The values of $\lambda$ that verify \eqref{eig_problem1}
are known as eigenvalues of $A$ and the corresponding vectors $x\in\mathbb{R}^{n}$ are the eigenvectors.
A simple way to compute extreme eigenvalues of $A$ is by minimizing (or maximizing) the associated Rayleigh quotient
\begin{equation}
r(x) = \frac{x^{\top}Ax}{x^{\top}x}.  \label{eig_problem2}
\end{equation}
This function is a continuously differentiable map $r:\mathbb{R}^{n}-\{0\}\to \mathbb{R}$, whose gradient is given by $\nabla r(x) = \frac{2}{||x||_2^{2}}( Ax - r(x)x )$. It is clear that any eigenvector $x$ and its corresponding eigenvalue $\lambda$ satisfy that $r(x) = \lambda$, and thus in that case $x$, is a critical point of $r(\cdot)$, i.e., $\nabla r(x) = 0$. By noting that $\nabla r(x) = 0$ if, and only if $x$ is a solution of the following nonlinear system of equations
\begin{equation}
G(x) \equiv Ax - r(x)x = 0,  \label{eig_problem3}
\end{equation}
we can address directly this eigenvalue problem by using SANE method. On the other hand, by introducing the
constraint $||x||_2 = 1$, the nonlinear system $G(x) = 0$ can be cast to the following tangent vector field,
$F:S^{n-1}\to \mathcal{T}S^{n-1}$ defined on a Riemannian manifold (the unit sphere)
\begin{equation}
F(x) \equiv Ax - (x^{\top}Ax)x = 0_x.  \label{eig_problem4}
\end{equation}
Note that for all $x\in S^{n-1}$, we have $x^{\top}F(x) = 0$. Therefore $F(\cdot)$ effectively maps point on the sphere to its corresponding tangent space, i.e., $F(\cdot)$ defines a tangent vector field.\\

For illustrative purposes, in this subsection we compare the numerical performance of the SANE method solving \eqref{eig_problem3}, versus
its generalized Riemannian version RSANE solving \eqref{eig_problem4}. To do this, we consider 40 instances of symmetric and positive definite sparse matrices taken from the UF sparse Matrix collection \cite{davis2011university}\footnote{The SuiteSparse Matrix Collection tool--box is available in \url{https://sparse.tamu.edu/}}, which contains several matrices that arise in real applications. For this experiment, we stop both algorithms when is founded a vector $x_k$ that satisfies the inequality $||G(x_k)||_2< TOL$, where $TOL = \max\{\epsilon, \epsilon ||G(x_0)||_2\}$, with $\epsilon = 2$e-5. Note that for our proposed algorithm $F(x_k) = G(x_k)$, for all $x_k$ generated by RSANE, since our proposal preserves the constraint $||x||_2 = 1$, hence this stop criterion is well defined for our algorithm. The starting vector $x_0\in\mathbb{R}^{n}$ was generated by $x_0 = v/\sqrt{n}$, where $v = (1,1,\ldots,1)^{\top}$.\\

The numerical results associated to this experiments are summarized in Table \ref{tab:1}. In this table, we report the number of iteration (Nitr); the CPU--time in seconds (Time); the residual value $\min\left\{\frac{||G(x_{*})||_2}{||x_0||_2},||G(x_{*})||_2\right\}$ (NrmF), where $x_{*}$ denotes the estimated solution obtained by each algorithm; and the number of function evaluations (Nfe). In the case of SANE this is the number of times that SANE evaluates the $G(\cdot)$ map, while for our RSANE, Nfe denotes the number of times that it evaluates the functional $F(\cdot)$.\\

We observe, from Table \ref{tab:1}, that the RSANE algorithm is a robust option to solve linear eigenvalue problems. The performance of both SANE and RSANE varied over different matrices. We note that our RSANE is very competitive for large--scale problems. Indeed, the RSANE algorithm outperforms its Euclidean version (SANE) in general terms, since RSANE took on average a total of 3644.1 iterations and 31.167 seconds to solve all the 40 problems, while SANE took 34.079 seconds and 3864.3 iterations to solve all the instances on average. It is worth noting that although RSANE failed on problem \emph{bcsstk13}, it can become successful on \emph{bcsstk16}, while both methods take the maximum number of iterations on problems \emph{bcsstk27}, \emph{bcsstk28} and \emph{nasa4704}.\\

In Figure \ref{fig:1}, we plot the convergence history of SANE and RSANE methods for the particular instances ``1138\_bus'' and ``s3rmt3m1'', respectively. From this figure, we note that RSANE can converge faster than SANE, which is illustrated in Figure \ref{fig:1} (b). However the opposite conclusion is obtained from Figure \ref{fig:1}(a).\\

\begin{table}
\centering
\caption{Numerical results for eigenvalues computation.}
\label{tab:1}
\resizebox{13cm}{!}{\begin{tabular}{c c |c c c c| c c c c}
  \hline
  & & \multicolumn{4}{c|}{  SANE }& \multicolumn{4}{ c}{  RSANE  }\\
  \hline
  Name &	n &    Nfe  &   NrmF	& Nitr   & Time	  &Nfe  &   NrmF	& Nitr   & Time		 \\
  \hline
1138\_bus	&	1138	&	10621	&	9.70e-6	&	2791	&	0.284	&	14778	&	7.21e-6	&	3781	&	0.449	\\
af\_0\_k101	&	503625	&	2269	&	1.95e-5	&	692	    &	51.875	&	2074	&	1.99e-5	&	644	    &	49.273	\\
af\_1\_k101	&	503625	&	1798	&	2.00e-5	&	560	    &	39.915	&	1702	&	1.99e-5	&	451	    &	33.397	\\
af\_2\_k101	&	503625	&	2683	&	2.00e-5	&	796	    &	59.179	&	1421	&	1.93e-5	&	451	    &	34.259	\\
af\_3\_k101	&	503625	&	2207	&	1.93e-5	&	664	    &	49.175	&	1757	&	1.92e-5	&	541	    &	42.234	\\
af\_4\_k101	&	503625	&	2281	&	1.98e-5	&	694	    &	51.156	&	1758	&	1.85e-5	&	545	    &	43.613	\\
af\_5\_k101	&	503625	&	2651	&	1.76e-5	&	789	    &	58.730	&	1975	&	1.79e-5	&	607	    &	47.731	\\
af\_shell3	&	504855	&	305	    &	1.99e-5	&	119	    &	6.989	&	328	    &	1.79e-5	&	125	    &	8.026	\\
af\_shell7	&	504855	&	452	    &	1.14e-5	&	165	    &	10.204	&	363	    &	1.53e-5	&	133	    &	8.851	\\
apache1	    &	80800	&	15377	&	1.88e-5	&	3992	&	14.466	&	34224	&	1.54e-5	&	8434	&	34.905	\\
apache2	    &	715176	&	58685	&	1.93e-5	&	13970	&	801.886	&	52821	&	6.89e-6	&	12516	&	785.019	\\
bcsstk10	&	1086	&	2450	&	1.30e-5	&	737	    &	0.063	&	1954	&	1.87e-5	&	603	    &	0.054	\\
bcsstk11	&	1473	&	883	    &	1.76e-5	&	302	    &	0.034	&	5421	&	2.00e-5	&	1501	&	0.199	\\
bcsstk13	&	2003	&	19254	&	1.98e-5	&	3293	&	1.372	&	88608	&	3.08e-4	&	15000	&	6.523	\\
bcsstk14	&	1806	&	76	    &	7.15e-6	&	35	    &	0.006	&	4674	&	1.94e-5	&	1341	&	0.283	\\
bcsstk16	&	4884	&	105001	&	7.49e+0	&	15000	&	24.274	&	1430	&	1.41e-5	&	458	    &	0.386	\\
bcsstk27	&	1224	&	64534	&	3.43e-4	&	15000	&	2.937	&	63935	&	2.11e-4	&	15000	&	3.064	\\
bcsstk28	&	4410	&	62724	&	2.32e-4	&	15000	&	13.300	&	63527	&	1.83e-4	&	15000	&	14.293	\\
bcsstk36	&	23052	&	54355	&	2.00e-5	&	12861	&	76.108	&	26406	&	1.99e-5	&	6561	&	38.742	\\
bcsstm12	&	1473	&	11960	&	1.94e-5	&	3017	&	0.295	&	8632	&	1.99e-5	&	2268	&	0.240	\\
bcsstm23	&	3134	&	14536	&	1.98e-5	&	3670	&	0.500	&	14820	&	2.00e-5	&	3735	&	0.652	\\
bcsstm24	&	3562	&	8255	&	2.00e-5	&	2228	&	0.312	&	6285	&	1.85e-5	&	1721	&	0.267	\\
cfd1	    &	70656	&	2693	&	2.00e-5	&	801	    &	6.265	&	2201	&	1.83e-5	&	659	    &	5.238	\\
ex15	    &	6867	&	197	    &	7.92e-6	&	76	    &	0.027	&	9578	&	1.97e-5	&	2606	&	1.334	\\
fv1	        &	9604	&	371	    &	1.81e-5	&	136	    &	0.061	&	474	    &	1.85e-5	&	170	    &	0.0949	\\
fv2	        &	9801	&	369	    &	1.54e-5	&	138	    &	0.066	&	334	    &	1.51e-5	&	127	    &	0.0637	\\
fv3	        &	9801	&	506	    &	1.82e-5	&	183	    &	0.077	&	371	    &	1.52e-5	&	136	    &	0.0664	\\
Kuu	        &	7102	&	5050	&	3.28e-6	&	1479	&	1.533	&	2674	&	4.82e-6	&	819	    &	0.9281	\\
mhd4800b	&	4800	&	1959	&	1.98e-5	&	605	    &	0.146	&	1784	&	1.96e-5	&	561	    &	0.1363	\\
msc23052	&	23052	&	43296	&	1.99e-5	&	10410	&	62.815	&	37779	&	1.99e-5	&	9217	&	56.9276	\\
Muu	        &	7102	&	29	    &	1.78e-5	&	13	    &	0.006	&	28	    &	1.70e-5	&	12	    &	0.0052	\\
nasa4704	&	4704	&	63663	&	8.21e-5	&	15000	&	7.418	&	63440	&	5.79e-5	&	15000	&	8.264	\\
s1rmq4m1	&	5489	&	9000	&	2.00e-5	&	2411	&	2.394	&	11052	&	1.99e-5	&	2933	&	3.2589	\\
s1rmt3m1	&	5489	&	23019	&	1.98e-5	&	5895	&	4.566	&	13688	&	1.93e-5	&	3592	&	3.1067	\\
s2rmq4m1	&	5489	&	9839	&	2.00e-5	&	2564	&	2.384	&	10292	&	1.96e-5	&	2683	&	2.774	\\
s2rmt3m1	&	5489	&	17422	&	2.00e-5	&	4437	&	3.504	&	21669	&	1.93e-5	&	5385	&	4.9953	\\
s3rmq4m1	&	5489	&	6063	&	1.85e-5	&	1710	&	1.500	&	5422	&	1.20e-5	&	1517	&	1.4757	\\
s3rmt3m1	&	5489	&	12250	&	1.76e-5	&	3295	&	2.467	&	6331	&	1.52e-5	&	1764	&	1.3742	\\
s3rmt3m3	&	5357	&	13090	&	1.89e-5	&	3461	&	2.569	&	9911	&	1.83e-5	&	2669	&	2.1663	\\
sts4098	    &	4098	&	22153	&	1.89e-5	&	5583	&	2.335	&	17664	&	1.93e-5	&	4499	&	2.0278	\\
\hline
\end{tabular}}
\end{table}

\begin{figure}
  \centering
  \begin{center}
  \subfigure[1138\_bus]{\includegraphics[width=6cm]{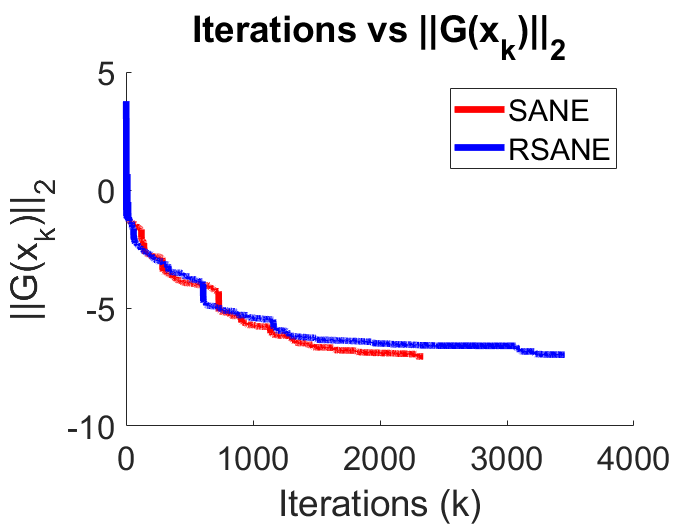}}
  \subfigure[s3rmt3m1]{\includegraphics[width=6cm]{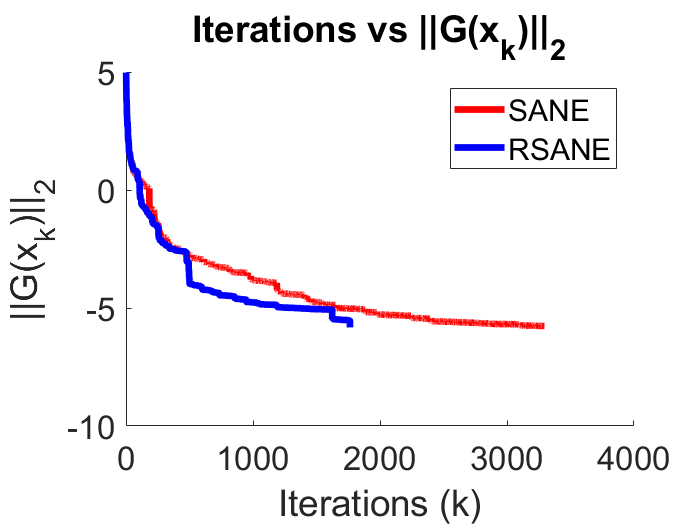}}
  \end{center}
  \caption{Convergence history of SANE and RSANE for the instances ``1138\_bus'' and ``s3rmt3m1''. The y--axis is in logarithmic scale.}
  \label{fig:1}
\end{figure}

\subsection{A nonlinear eigenvalue problem on the Stiefel manifold}
As a second experiment, we investigate the performance our RSANE method applied to deal with the following Stiefel manifold constrained nonlinear eigenvalue problem
\begin{subequations}\label{KKTs}
    \begin{eqnarray}
    H(X)X - X\Lambda    &=&  0, \label{eq1_NEP}\\
     X^{\top}X - I_p    &=&  0. \label{eq2_NEP}
    \end{eqnarray}
\end{subequations}
where $H(X) = L + \mu\,\verb"ddiag"(L^{\dag}\rho(X))$, where $L^{\dag}\in\mathbb{R}^{n\times n}$ is the pseudo--inverse matrix of the discrete Laplacian operator $L$, $\mu>0$ and $\rho(X) = \verb"diag"(XX^{\top})$, here $\verb"diag"(M)\in\mathbb{R}^{n}$ denotes the vector containing the diagonal elements of the matrix $M\in\mathbb{R}^{n\times n}$. Observe that $\Lambda\in\mathbb{R}^{p\times p}$ can be seen as a block of eigenvalues of the nonlinear matrix $H(X)$. In fact, in the special case that the operator $H(X)$ is constant and $p=1$, this problem is reduced to the linear eigenvalue problem \eqref{eig_problem1}. This kind of nonlinear eigenvalue problem appear frequently in total energy minimization in electronic structure calculations \cite{cedeno2018projected,saad2010numerical}. Pre--multiplying by $X^{\top}$ both sides of \eqref{eq1_NEP} and using \eqref{eq2_NEP} we obtain $\Lambda = X^{\top}H(X)X$, so substituting this result in \eqref{eq1_NEP}, we obtain the following tangent vector field defined on the Stiefel manifold
\begin{equation}\label{NEP_F}
  F(X) \equiv H(X)X - XX^{\top}H(X)X = 0_X,
\end{equation}
where $F:St(n,p)\to \mathcal{T}St(n,p)$.\\

In Table \ref{tab:2} we report the results obtained by running the RSANE and CGPR methods on the problem of finding a zero of \eqref{NEP_F}, with six possible choices for the pair $(n,p)$, obtained by varying $n$ and $p$ in $\{100,500,1000\}$ and $\{10,50\}$, respectively. For comparison purposes, we repeat our experiments over 30 different random generated starting points $X_0\in St(n,p)$ for each pair of $(n,p)$ and report the averaged number of iteration (Nitr), the averaged number of the evaluation of the functional $F(\cdot)$ (Nfe), the averaged total computing time in seconds (Time) and the averaged residual (NrmF) given by $\frac{1}{30}\sum_{i=1}^{30}||F(x_{*}^{i})||_F$, where $x_{*}^{i}$ denotes the estimated solution obtained by the respective algorithm for the $i$--th starting point, for all $i\in\{1,2,\ldots,30\}$. In addition, for all the experiments we fix $\mu = 1$, $\epsilon = 1$e-4 as the tolerance for the termination rule based on residual norm $||F(X_k)||_F< \epsilon$. In this table, RSANE\_polar and RSANE\_qr denotes our Algorithm \ref{Alg1} using the retractions defined in \eqref{polarR_stiefel} and \eqref{qrR_stiefel}, respectively. Similarly, CGPR\_polar and CGPR\_qr denotes the Riemannian Derivative--Free conjugate gradient Polak--Ribi\'ere--Polyak method developed in \cite{yao2020riemannian} using the retractions \eqref{polarR_stiefel} and \eqref{qrR_stiefel}, respectively.\\

As shown in Table \ref{tab:2}, RSANE is superior to the Riemannian conjugate gradient method solving nonlinear eigenvalues problems for different choices of $(n,p)$. In particular, we note that for problems with $p = 50$, our proposal basically converges in half of the iterations that the CGPR takes to reach the desired precision.\\

\begin{table}
\centering
\caption{Numerical results for nonlinear eigenvalue problems.}
\label{tab:2}
\resizebox{13cm}{!}{\begin{tabular}{c|c c c c| c c c c}
  \hline
  Method &	    Nitr  &   Time	& NrmF   & Nfe	  &Nitr	    &Time	  &NrmF   & Nfe		 \\
  \hline
  & \multicolumn{4}{c|}{  $(n,p) =  (100,10)$  }& \multicolumn{4}{ c}{  $(n,p) =     (100,50) $  }\\
  \hline
RSANE-polar	&	70.0	&	0.022	&	7.85e-6	&	167.2	&	632.9	&	0.898	&	2.10e-4	&	2170.1	\\
RSANE-qr	&	67.8	&	0.016	&	6.87e-6	&	161.9	&	620.6	&	0.505	&	4.10e-4	&	2126.9	\\
CGPR-polar	&	79.0	&	0.029	&	8.30e-5	&	231.3	&	1209.4	&	2.747	&	1.46e-3	&	6586.4	\\
CGPR-qr	    &	78.2	&	0.024	&	8.05e-5	&	229.3	&	996.8	&	1.296	&	3.30e-4	&	5451.3	\\
  \hline
  & \multicolumn{4}{c|}{  $(n,p) =   (500,10)$  }& \multicolumn{4}{c}{$(n,p) =   (500,50)  $  }\\
  \hline
RSANE-polar	&	65.0	&	0.200	&	8.10e-6	&	151.6	&	311.6	&	2.048	&	7.79e-6	&	955.4	\\
RSANE-qr	&	68.0	&	0.211	&	7.43e-6	&	160.0	&	311.5	&	1.928	&	8.79e-6	&	954.6	\\
CGPR-polar	&	74.3	&	0.284	&	8.14e-5	&	209.3	&	737.5	&	7.757	&	8.86e-5	&	3601.4	\\
CGPR-qr	    &	73.6	&	0.280	&	8.17e-5	&	210.2	&	767.1	&	7.583	&	9.12e-5	&	3731.1	\\
  \hline
  & \multicolumn{4}{c|}{  $(n,p) =   (1000,10)$  }& \multicolumn{4}{c}{$(n,p) =   (1000,50)  $  }\\
  \hline
RSANE-polar	&	68.9	&	0.628	&	7.55e-6	&	159.8	&	316.3	&	5.272	&	8.12e-6	&	968.9	\\
RSANE-qr	&	66.8	&	0.604	&	6.86e-6	&	153.4	&	321.3	&	5.294	&	8.03e-6	&	987.6	\\
CGPR-polar	&	73.5	&	0.820	&	7.57e-5	&	208.5	&	790.3	&	21.224	&	9.36e-5	&	3880.0	\\
CGPR-qr	    &	77.4	&	0.860	&	6.87e-5	&	219.6	&	881.2	&	23.505	&	9.50e-5	&	4347.7	\\
  \hline
\end{tabular}}
\end{table}

\subsection{Joint diagonalization on the oblique manifold}
In this subsection we analyze the numerical behaviour of our RSANE method solving the nonlinear equation based on the tangent vector field obtained from the Riemannian gradient of the scalar function $\mathcal{F}:\mathcal{OB}(n,p)\to \mathbb{R}$, defined by
\begin{equation}
\mathcal{F}(X) = \sum_{i=1}^{N} ||\verb"off"(X^{\top}C_iX)||_F^{2}, \label{JDP}
\end{equation}
where $\verb"off"(W) = W - \verb"ddiag"(W)$ and $C_i\in\mathbb{R}^{n\times n}$ are given symmetric matrices. The minimization of this function on $\mathcal{OB}(n,p)$ is frequently used to perform independent component analysis, see \cite{absil2006joint}. It is well--known that the problem of minimizing \eqref{JDP} is closely related to the problem of finding a zero to the Riemannian gradient of the function $\mathcal{F}(\cdot)$, which is given by
\begin{equation}
\nabla_{\mathcal{OB}(n,p)}\mathcal{F}(X) = \sum_{i=1}^{N} 4C_iX\verb"off"(X^{\top}C_iX), \label{JDP2}
\end{equation}
for details see Section 3 in \cite{absil2006joint}. In order to study the numerical performance of our RSANE compared with the CGPR method, we consider the following tangent vector field
\begin{equation}
F(X) \equiv \sum_{i=1}^{N} 4C_iX\verb"off"(X^{\top}C_iX) = 0_X, \label{JDP3}
\end{equation}
which is obtained form the Riemannian gradient of $\mathcal{F}(\cdot)$.\\

Now, we present a computational comparison considering the RSANE and CGPR \cite{yao2020riemannian} methods on the solution of the tangent vector field \eqref{JDP3}. To do this, we set $N = 5$, and study the numerical behavior of both methods for the pairs of values $(n,p) = (500,100)$ y $(n,p) = (1000,100)$. For each pair $(n,p)$, we repeat 30 independent runs of RSANE and CGPR, generating the symmetric matrices $C_i\in\mathbb{R}^{n\times n}$ as follows
$$  C_i = D + B_i + B_i^{\top},   $$
where $D\in\mathbb{R}^{n\times n}$ was a diagonal matrix, whose diagonal elements are given by $d_{ii} = \sqrt{n+i}$, for all $i\in\{1,2,\ldots,n\}$; and the $B_i's\in\mathbb{R}^{n\times n}$ were randomly generated matrices, whose entries follow a standard normal distribution. This particular structure of the $C_i$'s matrices was taken from \cite{zhu2017riemannian}. The starting point $X_0\in \mathcal{OB}(n,p)$ was randomly generated using the following Matlab commands
$$X_0 = \verb"zeros"(n, p); \quad M = \verb"randn"(n,p)\quad  \textrm{and} \quad  X_0(:,i) = M (:,i)/\verb"norm"(M (:,i)),$$
for all $i = 1, ..., p$. For this comparison, we use $\epsilon = 1$e-$5$ as the tolerance for the stopping rule based on the residual norm $||F(X_k)||_F < \epsilon $.\\

The mean of number of iteration, number of evaluation of $F(\cdot)$, total computational time in seconds, and the residual ``NrmF'' defined as in the previous subsection, are reported in Table \ref{tab:3}. In addition, in Figure \ref{fig:2} we plot the numerical behavior associated to the average residual curve $||F(X_k)||_F$ throughout the iterations, for all the methods and for each pairs $(n,p)$.\\

From Table \ref{tab:3} and Figure \ref{fig:2}, we can see that both RSANE and CGPR can find an approximation of the solution of problem \eqref{JDP3} with the pre--established precision for the residual norm, in the two cases of $(n,p)$ considered. In addition, we clearly observe that RSANE performs better than CGPR in terms of the mean value of iterations and CPU--time.

\begin{table}
\centering
\caption{Numerical results for joint diagonalization Riemannian gradient vector field.}
\label{tab:3}
\begin{tabular}{c c c c c}
  \hline
  Method &	    Nitr  &   Time	& NrmF   & Nfe	 	 \\
  \hline
  \multicolumn{5}{c}{  $(n,p) =  (500,100)$  }\\
  \hline
  RSANE	&	157.8	&	3.950	&	6.85e-6	&	390.0	\\
  CGPR	&	238.8	&	7.589	&	7.84e-6	&	741.8	\\
  \hline
  \multicolumn{5}{c}{  $(n,p) =  (1000,100)$  }\\
  \hline
  RSANE	&	68.1	&	4.894	&	6.50e-6	&	142.5	\\
  CGPR	&	153.4	&	14.342	&	7.59e-6	&	416.1	\\
  \hline
\end{tabular}
\end{table}

\begin{figure}
  \centering
  \begin{center}
  \subfigure[Iterations vs $||F(X_k)||_F$]{\includegraphics[width=6cm]{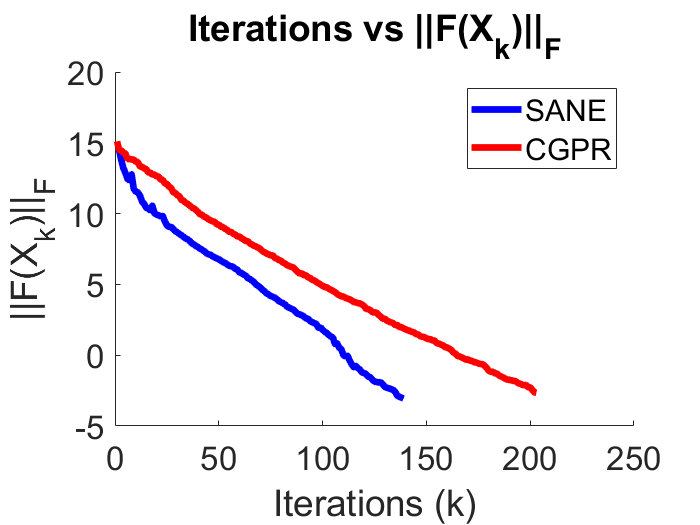}}
  \subfigure[Iterations vs $||F(X_k)||_F$]{\includegraphics[width=6cm]{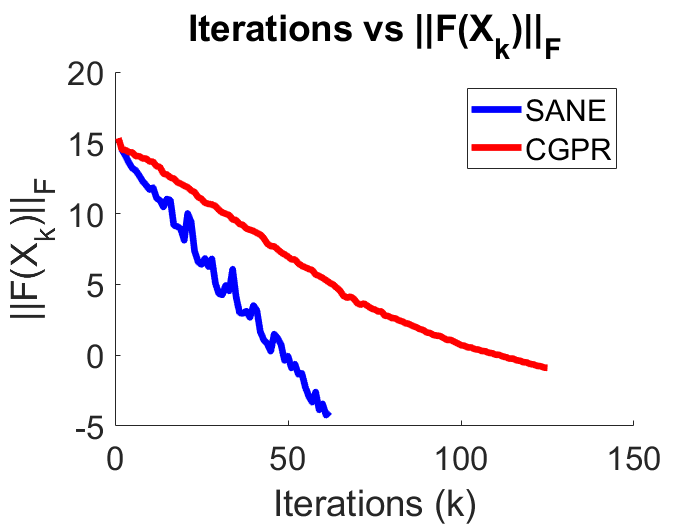}}
  \end{center}
  \caption{Averaged convergence history of RSANE and CGPR, from the same initial point, for the joint diagonalization Riemannian gradient vector field. On the left, $(n,p,N) = (500,100,5)$, and on the right $(n,p,N) = (1000,100,5)$. The y--axis is in logarithmic scale.}
  \label{fig:2}
\end{figure}

\section{Concluding Remarks}
\label{sec:6}
In this paper, a Riemannian residual approach for finding a zero of a tangent vector field is proposed. The new approach can be seen
as an extended version of the SANE method developed in \cite{cruz2003nonmonotone}, for the solution of large--scale nonlinear systems
of equations. Since the proposed method systematically uses the tangent vector field for building the search direction, RSANE is very easy
to implement and has low storage requirements, which is suitable for solving large--scale problems. In addition, our proposal uses a modification
of the Riemannian Barzilai--Borwein step--sizes introduced in \cite{iannazzo2018riemannian}, combined with the Zhang--Hager globalization
strategy \cite{zhang2004nonmonotone}, in order to guarantee the convergence of the associated residual sequence.\\

The preliminary numerical results show that RSANE performs efficiently, dealing with several tangent vector fields considering both real and simulated data, and
different Riemannian manifolds. In particular, RSANE is competitive against its Euclidean version (SANE), solving large--scale eigenvalue problems. Additionally, RSANE outperforms the derivative-free conjugate gradient algorithm recently published in \cite{yao2020riemannian}, on two considered matrix manifolds.


\bibliographystyle{plain}
\bibliography{references}

\end{document}